\documentclass[10pt]{article}
\usepackage[pagebackref,colorlinks=true]{hyperref}
\usepackage{amssymb, amsfonts, amsthm, mathrsfs}
\usepackage{graphicx, xcolor}
\usepackage[all]{xy}
\usepackage[leqno]{amsmath}

\theoremstyle{definition}
\newtheorem{thm}{Theorem}[section]

\newtheorem{prop}[thm]{Proposition}

\newtheorem{rmk}[thm]{Remark}
\theoremstyle{definition}
\newtheorem{ex}[thm]{Example}

\makeatletter

\def\id{\mathrm{id}}

\def\tr{\mathrm{Tr}}
\def\ccc{\mathbb{C}}

\def\zz{\mathbb{Z}}
\def\rr{\mathbb{R}}
\def\pp{\mathbb{P}}

\def\oo{\mathbb{O}}

\def\clo{\mathcal{O}}

\def\frI{\mathfrak{I}}
\def\pt{\partial}
\def\bpt{\overline{\pt}}

\def\ud{\mathrm{d}}

\def\bz{\overline{z}}

\def\bphi{\overline{\varphi}}
\def\bzeta{\overline{\zeta}}
\def\brho{\overline{\rho}}

\def\ima{\mathrm{Im}}

\makeatother

\title{A Construction of Non-K\"ahler Calabi-Yau Manifolds and New Solutions to the Strominger System}
\author{Teng Fei}

\begin{document}
\maketitle{}

\section{Introduction}

This paper is a follow-up of an observation made in \cite{fei2015b}. Nevertheless, I would like to say a few words about the general background first.

We call a compact complex manifold a Calabi-Yau, to be K\"ahler or not, if its canonical bundle is holomorphically trivial. Calabi-Yau manifolds, especially Calabi-Yau 3-folds, have been extensively studied ever since Yau's solution to the Calabi conjecture \cite{yau1977, yau1978}. Among many other enthralling problems, motivated from either geometry or string theory, there stands out the moduli problem. There exists a vast collection of literature devoted to understanding the moduli space of Calabi-Yau manifolds. We refer to the beautiful survey paper \cite{yau2009} and the references therein for more details. In our context, we would like to single out what is usually called the ``Reid's fantasy''. In \cite{reid1987}, Reid made the wild conjecture that all the reasonably nice compact 3-folds with trivial canonical bundle (including non-K\"ahler ones necessarily) can be connected with each other via conifold transitions. Reid's idea was supported by the work of Candelas-de la Ossa \cite{candelas1990}, where explicit Ricci-flat K\"ahler metrics are found on both deformed and resolved conifolds and their limiting behaviors are analyzed.

In order to globalize Candelas-de la Ossa's result, it is natural to ask the question that how canonical we can choose Hermitian metrics on compact non-K\"ahler Calabi-Yau manifolds. The first way to deal with this problem is to understand how to generalize the Calabi-Yau theorem in the non-K\"ahler setting. In particular, the balanced version of Gauduchon conjecture remains to be solved. A consequence of this conjecture is that on complex balanced manifolds with trivial canonical bundle one can always find a balanced metric within a suitable cohomology class such that its first Bott-Chern form vanishes. Progress was made by Sz\'ekelyhidi-Tosatti-Weinkove \cite{szekelyhidi2015} in this direction.

A second approach to this problem is to solve the Strominger system. This is a system of PDEs proposed by Strominger \cite{strominger1986} in the study of heterotic strings with torsion which we shall formulate. Let $(X^3,g,J)$ be an Hermitian 3-fold (not necessarily K\"ahler) with holomorphically trivial canonical bundle and let $\Omega$ be a nowhere-vanishing holomorphic $(3,0)$-form on $X$. We denote the positive $(1,1)$-form associated with $g$ by $\omega$ and the curvature form of $(T_\ccc X,g)$ by $R$. In addition, let $(E,h)$ be a holomorphic vector bundle with metric over $X$ and let $F$ be its curvature form. The Strominger system consists of the following equations:
\begin{eqnarray}
\label{hym}F\wedge\omega^2=0,\quad F^{0,2}=F^{2,0}=0,\\
\label{ac}\sqrt{-1}\pt\bpt\omega=\frac{\alpha'}{4}(\tr(R\wedge R)-\tr(F\wedge F)),\\
\label{cb}\ud(\|\Omega\|_\omega\cdot\omega^2)=0.
\end{eqnarray}
Equations (\ref{hym}), (\ref{ac}) and (\ref{cb}) are known as the Hermitian-Yang-Mills equation, the anomaly cancellation equation and the conformally balanced equation respectively.

If we assume that $\omega$ is K\"ahler and take $E=T_\ccc X$, the anomaly cancellation (\ref{ac}) is automatic. It follows that the whole system is reduced to an equation requiring $g$ to be Ricci-flat. In this sense, the Strominger system generalizes the complex Monge-Amp\`ere equation used in K\"ahler geometry. Therefore we may regard solutions to Strominger system as canonical metrics, even on a non-K\"ahler background. Actually, the first method mentioned above can be incorporated into this picture. It is well-known that for $X$ with trivial canonical bundle, the conformally balanced condition (\ref{cb}) is equivalent to that the restricted holonomy of $X$ with respect to the Strominger-Bismut connection is contained in $SU(3)$, see \cite[Section 2]{strominger1986} and \cite[Lemma 3.1]{li2005}. However, this is very far from requiring $\Omega$ to be parallel under the Strominger-Bismut connection. It turns out the latter condition is equivalent to that $\|\Omega\|_\omega$ is a constant and $\omega$ is a balanced metric, i.e., solving the balanced version of Gauduchon conjecture.

The Strominger system is very hard to solve in general. The reason lies in the fact that the anomaly cancelation equation (\ref{ac}) is an equation of 4-forms. After certain perturbative solutions described in \cite{strominger1986}, the first smooth irreducible solutions for $U(4)$ or $U(5)$ principal bundles over K\"ahler Calabi-Yau's were due to Li and Yau \cite{li2005}. Later, a set of genuine non-K\"ahler solutions were obtained by Fu and Yau \cite{fu2008}. A great deal of work has been done in recent years, see \cite{fernandez2009}, \cite{fu2009}, \cite{andreas2012}, \cite{anderson2014}, \cite{ossa2014}, \cite{fei2015}, \cite{garciafernandez2015} and the references therein.

Following \cite{michelsohn1982}, we say an Hermitian metric $\omega$ on a complex $n$-fold is \emph{balanced} if $\ud(\omega^{n-1})=0$. As a consequence of (\ref{cb}), $X$ must support a balanced metric. Thanks to the work of \cite{clemens1983}, \cite{friedman1986} and \cite{fu2012}, there are lots of non-K\"ahler Calabi-Yau 3-folds with balanced metric which can be obtained by taking a sequence of conifold transitions starting from a projective Calabi-Yau 3-fold. Other techniques like branched double covering \cite{lin2014} may also be useful and we refer to the survey paper \cite{tosatti2015} for more constructions. However, it seems that for most of these constructions, the Strominger system is way too hard to attack. As far as the author knows, the only successes were made on those $T^2$-bundles over K3 surfaces constructed by \cite{goldstein2004} and quotients of various Lie groups over lattices.

This paper is divided into two parts. In Part \ref{p1}, following the observation made in \cite{fei2015b}, we provide a new way to find non-K\"ahler manifolds which generalizes the classical construction of Calabi \cite{calabi1958} and Gray \cite{gray1969}. In particular, we obtain a series of non-K\"ahler Calabi-Yau 3-folds with natural balanced metric. These 3-folds are holomorphic fiber bundles over Riemann surfaces of genus $g\geq3$ with hyperk\"ahler fibers. For Part \ref{p2}, we make use of this fibration structure to write down a suitable ansatz for the Strominger system. By solving an algebraic equation, we obtain explicit degenerate solutions to the Strominger system with $F=0$. Along the way we also prove that our models do not admit any pluriclosed metrics, answering a question of Fu-Wang-Wu \cite[Section 1]{fu2013}. Finally we make a discussion about the geometry of degeneracy loci.

\subsection*{Acknowledgement}
The author would like to thank Prof. Shing-Tung Yau and Prof. Victor Guillemin for their constant encouragement and help. The conversation with Prof. Claude LeBrun on theta characteristics is extremely inspiring. The author is also indebted to various helpful discussions and correspondence with Li-Sheng Tseng, Bao-Sen Wu, Peter Smillie and Cheng-Long Yu. Finally the author thanks Valentino Tosatti and Eirik Svanes for pointing out several important references.

\part{Geometric Construction}\label{p1}

\section{The Classical Construction of Calabi and Gray}\label{s2}

For any oriented immersed hypersurface $X$ in $\rr^7$, Calabi \cite{calabi1958} discovered that $X$ automatically admits an almost complex structure $J_0$ from the following construction. Let us identify $\rr^7$ with $\ima(\oo)$, the space of purely imaginary octonions. There is a cross product $\times$ defined on $\ima(\oo)$ which can be expressed by \[a\times b=a\cdot b+\langle a,b\rangle,\] where $\cdot$ is the octonion multiplication and $\langle-,-\rangle$ is the standard inner product on $\oo$. For any point $p\in X$, the almost complex structure $(J_0)_p:T_pX\to T_pX$ is defined by \[(J_0)_p(v)=\nu_p\times v,\quad\forall v\in T_pX,\] where $\nu_p$ is the unit positive normal of $M$ at $p$ with respect to the standard metric and orientation. Using properties of cross product on $\ima(\oo)$, it is not hard to check that $J_0^2=-\id$ and we get an almost complex structure.

Calabi also derived the condition for $J_0$ to be integrable. In particular, he proved that if $\Sigma$ is an oriented immersed minimal surface in $\rr^3$, then the almost complex structure induced on $X=\Sigma\times\rr^4\looparrowright \rr^3\times\rr^4=\ima(\oo)$ is integrable. Using such construction, he gave the first example that the Chern classes of a complex manifold is not determined by its underlying smooth structure, answering a question of Hirzebruch.

Calabi's construction was later generalized by Gray \cite{gray1969} to manifolds with vector cross product. That is, one can replace $\rr^7$ by any manifolds with a $G_2$-structure. In particular, if $\Sigma_g$ is an oriented minimal surface of genus $g$ in flat $T^3$ and $M$ is a compact hyperk\"ahler 4-manifold (either $T^4$ or a K3 surface), then the induced $J_0$ on $X=\Sigma_g\times M\looparrowright T^3\times M$ is integrable. Notice that though diffeomorphically $X$ is a product, the holomorphic structure on $X$ is twisted, as long as $\Sigma_g$ is not totally geodesic, i.e., a flat subtorus. Such minimal surfaces $\Sigma_g$ exist for and only for every $g\geq3$, see \cite[Corollary 3.1, Corollary 10.1]{meeks1990} and \cite[Theorem 1]{traizet2008}.

It was further observed in \cite[Theorem 5.1]{fei2015b} that these $X$ are actually compact non-K\"ahler Calabi-Yau's. The argument goes as follows. As the fundamental 3-form $\varphi$ of the $G_2$-manifold $T^3\times M$ is closed, its restriction on $X$ is also closed. On the other hand, the restricted 3-form is the real part of a complex (3,0)-form $\Omega$. Using the fact $J_0$ is integrable, we conclude that $\Omega$ is holomorphic and therefore trivializes the canonical bundle of $X$.

There is also a natural induced metric on $X$ from the ambient $G_2$-manifold $T^3\times M$. It is classically known \cite{gray1969} that this metric $\omega$ is Hermitian and balanced. In addition, $\|\Omega\|_\omega$ is of constant length and therefore $\Omega$ is parallel under both the Chern connection and the Strominger-Bismut connection. Thus $(X,\omega)$ is a \emph{special balanced 3-fold}, using the terminology from \cite{fei2015b}.

Moreover, the fibration \[\pi:(\Sigma_g\times M,J_0)\to\Sigma_g\] is holomorphic with holomorphic sections of the form $\Sigma_g\times\{\textrm{pt}\}$.

\section{Explicit Calculation}

In this section, we lay the foundation for future calculations.

Let us first write down the complex structure $J_0$ explicitly. Let $e_1,e_2,e_3$ be an orthonormal basis of parallel vector fields on $T^3$ and let $e^1,e^2,e^3$ be the dual 1-forms. Fix $I,J,K$ a set of pairwise anti-commuting complex structures on the hyperk\"ahler manifold $M$, and denote the associated K\"ahler forms by $\omega_I$, $\omega_J$ and $\omega_K$ respectively. Let $G:\Sigma_g\to S^2\subset\rr^3$ be the Gauss map and write its components as \[G(z)=(\alpha(z),\beta(z),\gamma(z))\in\rr^3,\quad\forall z\in\Sigma_g.\]  Thus \[\nu(z)=\alpha(z)e_1+\beta(z)e_2+\gamma(z)e_3.\] Notice that the cross product on $\Sigma_g\times M$ is determined by the fundamental 3-form \[\varphi=e^1\wedge\omega_I+e^2\wedge\omega_J+e^3\wedge\omega_K-e^1\wedge e^2\wedge e^3.\footnote{It should be noted that the orientation induced on $T^3$ differs from the standard one by a sign.}\] It is not hard to see that
\begin{equation}\label{complex structure}
\begin{split}J_0e_1&=-\gamma e_2+\beta e_3,\\ J_0e_2&=\gamma e_1-\alpha e_3,\\ J_0e_3&=-\beta e_1+\alpha e_2,\\ J_0v&=\alpha Iv+\beta Jv+\gamma Kv,\end{split}
\end{equation}
for arbitrary vector field $v$ tangent to fibers of $\pi:X\to\Sigma_g$.

The action of $J_0$ on forms can be obtained easily as follow \[\begin{split}J_0e^1&=\gamma e^2-\beta e^3,\\ J_0e^2&=-\gamma e^1+\alpha e^3,\\ J_0e^3&=\beta e^1-\alpha e^2,\\J_0\omega_I&=(2\alpha^2-1)\omega_I+2\alpha\beta\omega_J+2\alpha\gamma\omega_K,\\ J_0\omega_J&=2\alpha\beta\omega_I+(2\beta^2-1)\omega_J+2\beta\gamma\omega_K,\\ J_0\omega_K&=2\alpha\gamma\omega_I+2\beta\gamma\omega_J+(2\gamma^2-1)\omega_K.\end{split}\] Denote by $\omega_0$ the induced metric on $X$ from $T^3\times M$, clearly,
\begin{equation}\label{metric}
\omega_0=\omega+\alpha\omega_I+\beta\omega_J+\gamma\omega_K,
\end{equation}
where $\omega$ is the induced K\"ahler metric on $\Sigma_g$.

Up to now we have not used the fact that $\Sigma_g$ is minimal. Let $f:D\to \Sigma_g\subset\rr^3$ given by \[(u,v)\mapsto(f_1(u,v),f_2(u,v),f_3(u,v))\] be an isothermal parametrization of $\Sigma_g$ compatible with its orientation. Let $z=u+iv$ and \[\varphi_j=\frac{\pt f_j}{\pt u}-i\frac{\pt f_j}{\pt v}\] for $j=1,2,3$. It is a well-known fact that $\Sigma_g$ is a minimal surface is equivalent to that $\varphi_j$ are holomorphic functions and \[\varphi_1^2+\varphi_2^2+\varphi_3^2=0.\] Setting \[2\lambda=\varphi_1\bphi_1+\varphi_2\bphi_2+\varphi_3\bphi_3,\] we can easily express $\alpha,\beta,\gamma$ as \[\begin{split}-2i\lambda\alpha&=\varphi_2\bphi_3-\varphi_3\bphi_2,\\ -2i\lambda\beta&=\varphi_3\bphi_1-\varphi_1\bphi_3,\\ -2i\lambda\gamma&=\varphi_1\bphi_2-\varphi_2\bphi_1.\end{split}\] Without too much effort, one can check that \[\varphi_1^{-1}\frac{\pt\alpha}{\pt\bz}=\varphi_2^{-1}\frac{\pt\beta}{\pt\bz}=\varphi_3^{-1}\frac{\pt\gamma}{\pt\bz}.\] We also have the relations
\begin{equation}\label{relation}
\begin{split}-i\frac{\pt\alpha}{\pt\bz}&=\beta\frac{\pt\gamma}{\pt\bz}-\gamma\frac{\pt\beta}{\pt\bz}\\ -i\frac{\pt\beta}{\pt\bz}&=\gamma\frac{\pt\alpha}{\pt\bz}-\alpha\frac{\pt\gamma}{\pt\bz},\\ -i\frac{\pt\gamma}{\pt\bz}&=\alpha\frac{\pt\beta}{\pt\bz}-\beta\frac{\pt\alpha}{\pt\bz}.\end{split}
\end{equation}

\section{Generalizations}

In this section, we will generalize the classical construction of Calabi and Gray. A first observation is that the recipe (\ref{complex structure}) used for producing the complex structure $J_0$ make sense for any hypercomplex manifold $M$. More precisely, let $N$ be a hypercomplex manifold of complex dimension $2n$, i.e., a smooth manifold of real dimension $4n$ endowed with three integrable complex structures $I$, $J$ and $K$ satisfying \[I^2=J^2=K^2=IJK=-\id.\] Let $\Sigma_g$ be an oriented minimal surface immersed in $T^3$ as before, then we can define an almost complex structure $J_0$ on $\Sigma_g\times N$ using exactly the formula (\ref{complex structure}). A straightforward calculation of Nijenhuis tensor shows that $J_0$ is integrable, due to the relations we derived in (\ref{relation}). If we further assume that $N$ is actually hyperk\"ahler, then the associated Hermitian forms $\omega_I$, $\omega_J$ and $\omega_K$ are closed. Notice that the naturally induced metric $\omega_0$ on $\Sigma_g\times N$ is still of the form (\ref{metric}). We see immediately that \[\omega_0^{2n}=(\alpha\omega_I+\beta\omega_J+\gamma\omega_K)^{2n}+2n\cdot \omega\wedge(\alpha\omega_I+\beta\omega_J+\gamma\omega_K)^{2n-1}\] is $\ud$-closed, hence $\omega_0$ is a balanced metric.

\begin{rmk}
Hypercomplex manifolds form a strictly larger category than hyperk\"ahler manifolds. For instance, in real dimension 4, besides $T^4$ and K3 surfaces, we should also include the Hopf surfaces into the hypercomplex list, see \cite{boyer1988}.
\end{rmk}

Our next observation relates the above construction to the twistor space of hypercomplex manifolds. Recall from \cite[Section 3(F)]{hitchin1987} and \cite[Section 1]{kaledin1998} that the twistor space of a hypercomplex manifold $N$ is constructed as follows. Parameterize $S^2=\{(\alpha,\beta,\gamma)\in\rr^3:\alpha^2+\beta^2+\gamma^2=1\}$ by $\zeta\in\ccc\pp^1$ via \[(\alpha,\beta,\gamma)= \left(\frac{1-|\zeta|^2}{1+|\zeta|^2},\frac{\zeta+\bzeta}{1+|\zeta|^2},\frac{i(\bzeta-\zeta)}{1+|\zeta|^2}\right).\footnote{The $\gamma$ component here differs from \cite{hitchin1987} by a sign. If we follow the convention that $IJ=K$ on vector fields, then $IJ=-K$ on 1-forms. With this understood, one detects a sign issue in \cite{hitchin1987}'s calculation. It turns out our formula above gives the right complex structure.}\] The twistor space $Z$ of $N$ is defined to be the manifold $Z=\ccc\pp^1\times N$ with the almost complex structure $\frI$ given by \begin{equation}\label{hypertwistor}
\frI=j\oplus(\alpha I_x+\beta J_x+\gamma K_x)
\end{equation}
at point $(\zeta,x)\in\ccc\pp^1\times N$, where $j$ is the standard complex structure on $\ccc\pp^1$ with holomorphic coordinate $\zeta$. It is a theorem of \cite{hitchin1987} and \cite{kaledin1998} that $\frI$ is integrable and the projection $p:Z\to\ccc\pp^1$ is a holomorphic fibration.

Recall that for an oriented minimal surface $\Sigma_g$ in $T^3$, the Gauss map $G:\Sigma_g\to\ccc\pp^1$ written in above coordinate \[z\mapsto\zeta(z)\] is holomorphic. Comparing the definition of $J_0$(\ref{complex structure}) and $\frI$(\ref{hypertwistor}), we conclude immediately that $J_0$ is exactly the pullback of $\frI$ by the Gauss map $G$. More precisely, we have the pullback square \[\xymatrix{(\Sigma_g\!\times\! N,J_0)\cong G^*\!Z\ar[r]^{\quad\quad\tilde{G}}\ar[d]_\pi & (Z,\frI)\ar[d]^p\\ \Sigma_g\ar[r]^G & \ccc\pp^1}\] All the maps in this diagram are holomorphic.

This is a very useful viewpoint in practice. For instance, it allows us to compute the first Chern class of $(\Sigma_g\times N,J_0)$. Let $F_\pi=\ker\ud\pi$ and $F_p=\ker\ud p$ respectively, and we have two short exact sequences of bundles \[\begin{split}&0\to F_\pi\to TG^*\!Z\xrightarrow{\ud\pi}\pi^*T\Sigma_g\to0,\\ &0\to F_p\to TZ\xrightarrow{~\ud p~}p^*T\ccc\pp^1\to 0.\end{split}\] From the pullback square, we see that \[F_\pi=\tilde{G}^*F_p,\] therefore we conclude that \[c_1(\Sigma_g\times N)=c_1(F_\pi)+\pi^*c_1(\Sigma_g)=\pi^*c_1(\Sigma_g)+\tilde{G}^*c_1(F_p).\] It was proved in \cite[Theorem 3.3]{hitchin1987} that $\wedge^2F_p^*\otimes p^*\clo(2)$ has a section which defines a holomorphic symplectic form on each fiber of $p$.\footnote{As long as we are only concerned with topology instead of holomorphic structure, this result is available to hypercomplex manifold as well.} It follows that \[c_1(F_p)=p^*c_1(\clo(2n)).\] Let $\mu$ be the positive generator of $H^2(\Sigma_g,\zz)$ and identify it with $\mu\otimes1\in H^*(\Sigma_g\times N,\zz)$, we conclude that \[c_1(\Sigma_g\times N)=(g-1)(2n-2)\mu,\] where we have used the simple fact that $\deg G=g-1$ as a consequence of Gauss-Bonnet.

We see immediately that when $n=1$, i.e., $N$ is a hypercomplex 4-manifold, then $c_1(\Sigma_g\times N)=0$. This is consistent with our result mentioned in Section \ref{s2}. We also conclude that $c_1(\Sigma_g\times N)$ can always be represented by a nonnegative class, which is good enough for us to prove the following theorem.

\begin{thm}{\label{nonk}}
$(\Sigma_g\times N,J_0)$ does not support any K\"ahler metric.
\end{thm}
To prove this theorem, we need to use Yau's generalized Schwarz lemma \cite{yau1978b}, which says
\begin{thm}(Yau, \cite[Theorem 2]{yau1978b})\\
Let $P_1$ be a complete K\"ahler manifold with nonnegative Ricci curvature and let $P_2$ be an Hermitian manifold with holomorphic bisectional curvature bounded from above my a negative constant. Then any holomorphic map from $P_1$ to $P_2$ must be a constant.
\end{thm}
Now we proceed to prove Theorem \ref{nonk}
\begin{proof}
Assume that $\Sigma_g\times N$ admits a K\"ahler metric, then by Yau's solution of Calabi conjecture \cite{yau1977, yau1978}, we may choose the K\"ahler metric to have nonnegative Ricci curvature since $c_1(\Sigma_g\times N)$ is the Ricci form. On the other hand, as $g\geq3$, we know that $\Sigma_g$ admits a hyperbolic metric. This contradicts with Yau's Schwarz lemma since $\pi:\Sigma_g\times N\to\Sigma_g$ is not a constant.
\end{proof}
\begin{rmk}
Yau's Schwarz lemma was later generalized by Tosatti \cite{tosatti2007} to the case that $P_1$ is a complete almost Hermitian manifold. In this case, one should use the 2nd Ricci curvature with respect to the so-called canonical connection, which coincides with the Chern connection in the integrable case. Our example of non-K\"ahler Calabi-Yau shows that one cannot replace the 2nd Ricci curvature by the 1st Ricci curvature.
\end{rmk}
We next prove that
\begin{thm}
$\Sigma_g\times N$ admits balanced metrics.
\end{thm}
\begin{proof}
An explicit balanced metric has been constructed at the beginning of this section in the case that $N$ is hyperk\"ahler. However, we will present a proof that is also applicable to the hypercomplex setting.

In \cite{tomberg2015}, Tomberg proved that $Z$ admits a balanced metric. If we pull-back this form back to $\Sigma_g\times N$ via $\tilde{G}$, then we get a nonnegative (1,1)-form $\omega_0$ whose $2n$-th power is closed. Notice that $G:\Sigma_g\to\ccc\pp^1$ is a branched cover, therefore $\omega_0^{2n}$, identified as a (1,1)-form via a volume form on $\Sigma_g\times N$, is degenerate at those fibers of $\pi$ over ramification points of $G$, only in the vertical direction.

To remedy this problem, for each ramification point $q\in\Sigma_g$, we construct a closed $(2n,2n)$-form $\lambda_q$ on $\Sigma_g\times N$ as following. Let $z$ be a local coordinate on $\Sigma_g$ with $z(q)=0$. Let $f_t$ be a real smooth function on $\Sigma_g$ such that $f_t$ is given by \[f_t(z)=(1+|z|^2)^t\] near $q$. Consider the closed real $(2n,2n)$-form \[\lambda_q=i\pt\bpt(f_t(\alpha\omega_I+\beta\omega_J+\gamma\omega_K)^{2n-1}).\] Evaluate it at $q$, we see that \[\lambda_q(q)=it\ud z\wedge\ud\bz\wedge(\alpha\omega_I+\beta\omega_J+\gamma\omega_K)^{2n-1} +i\pt\bpt(\alpha\omega_I+\beta\omega_J+\gamma\omega_K)^{2n-1}.\] Therefore for $t$ big enough, $\lambda_q(q)$ identified as a (1,1)-form, is strictly positive definite in the vertical direction.

Now let \[\lambda_0=\omega_0^{2n}+\sum_q c_q\lambda_q,\] where $c_q$ are positive constants. Our argument shows that if we choose $c_q$ close to 0, then $\lambda_0$ is a closed positive $(2n,2n)$-form on $\Sigma_g\times N$, hence we prove our theorem.

It should be noticed that the only fact about $G$ we used in the proof is that it is a branched covering.
\end{proof}

Now let us assume that $N$ is hyperk\"ahler and we replace $G:\Sigma_g\to\ccc\pp^1$ by an arbitrary holomorphic map $h:Y\to\ccc\pp^1$. Let $\widetilde{Y}$ be the total space of the pullback fibration $h^*Z$. Our computation above implies that \[K_{\widetilde{Y}}\cong K_Y\otimes h^*\clo(-2n),\] where $K_Y$ and $K_{\widetilde{Y}}$ represent the canonical bundle of $Y$ and $\widetilde{Y}$ respectively. Use an argument similar to \cite[Corollary 1]{tomberg2015}, we actually have proved
\begin{thm}\label{CY}
If $h:Y\to\pp^1$ satisfies
\begin{equation}\label{pullback}
K_Y\cong h^*\clo(2n),
\end{equation}
then $\widetilde{Y}=h^*Z$ has trivial canonical bundle. As long as $h$ is not a constant map, $\widetilde{Y}$ is non-K\"ahler.
\end{thm}
\begin{rmk}
A similar construction was used by LeBrun \cite{lebrun1999} for different purposes.
\end{rmk}

If (\ref{pullback}) is satisfied, then $L=h^*\clo(n)$ is a square root of $K_Y$, which corresponds to a spin structure on $Y$ according to Atiyah \cite{atiyah1971}. $L$ is known as a theta characteristic in the case that $Y$ is a curve. The minimal surface $\Sigma_g$ we considered is a special case of such construction with $n=1$. For $Y$ a curve and $n=1$, such an $h$ exists if and only if there is a theta characteristic $L$ on $Y$ such that $h^0(Y,L)\geq2$, i.e., $L$ is a vanishing theta characteristic.

We shall see that Theorem \ref{CY} is indeed a more general construction compared to Calabi-Gray. However, we do not have any natural metric on $\widetilde{Y}$ from this consideration.

\begin{ex}
Let $Y$ be a smooth $g=3$ curve and set $n=1$. It is well known, see \cite[Section 4]{gross2004} for instance, that $Y$ admits a vanishing theta characteristic if and only if $Y$ is hyperelliptic. Hyperelliptic genus 3 curves have a moduli of complex dimension 5, while it seems that the largest family of minimal genus 3 curves in $T^3$ we know is of real dimension 5 constructed in \cite[Theorem 7.1]{meeks1990}.
\end{ex}

\begin{ex}
For every hyperelliptic curves $Y$ of genus $g\geq3$, vanishing theta characteristics exist, so Theorem \ref{CY} can be used to construct non-K\"ahler Calabi-Yau 3-folds. However, it is a theorem of Meeks, see \cite[Theorem 3.3]{meeks1990}, that if $g$ is even, $Y$ can not be minimally immersed in $T^3$. From this we see that Theorem \ref{CY} yields examples cannot be covered by Calabi-Gray.

Actually, the set of genus $g$ curves with a vanishing theta characteristic form a divisor in the moduli space of genus $g$ curves. More refined results of this type can be found in \cite{harris1982} and \cite{teixidor1987}.
\end{ex}
\begin{ex}
If we allow $Y$ to be of higher dimension and $n$ to be greater than 1, then Theorem \ref{CY} can be used to construct lots more non-K\"ahler Calabi-Yau manifolds of higher dimension. For instance, we can take $Y\subset\ccc\pp^1\times\ccc\pp^r$ to be a smooth hypersurface of bidegree $(2n+2,r+1)$, then (\ref{pullback}) is satisfied, where $h$ is the restriction of the projection to $\ccc\pp^1$. There are also numerous examples of elliptic fibrations over $\ccc\pp^1$ without multiple fibers such that (\ref{pullback}) holds. These examples lead to simply-connected non-K\"ahler Calabi-Yau's of complex dimension not less than 4.
\end{ex}
\begin{rmk}
It has been known for many years that the Iitaka conjecture fails for general compact complex manifolds, see \cite[Remark 4]{nakamura1973} and \cite[Example 2]{magnusson2012}. It seems that all the counterexamples the author can find in literature involve non-K\"ahler manifolds with negative Kodaira dimension. On the other hand, the fibration $\pi:\widetilde{Y}\to Y$ we constructed in Theorem \ref{CY} has the property that $\kappa(\pi_y)=\kappa(\widetilde{Y})=0$ while $\kappa(Y)>0$, hence
\[\kappa(\widetilde{Y})<\kappa(Y)+\kappa(\pi_y),\]
violating the assertion of Iitaka conjecture.
\end{rmk}

\part{A Degenerate Solution to the Strominger System}\label{p2}

In Part \ref{p1} we constructed various non-K\"ahler Calabi-Yau manifolds with balanced metric. They are natural testing ground for heterotic strings. In Part \ref{p2} of this paper, we will only consider the simplest case, i.e., $X=\Sigma_g\times T^4$, where $\Sigma_g$ is a minimal genus $g$ Riemann surface in $T^3$ and $T^4$ is the real 4-torus with standard hyperk\"ahler structure. As we have seen in Section \ref{s2} that $X$ is a non-K\"ahler Calabi-Yau 3-fold whose natural metric is specially balanced. We will try to solve Strominger system on $X$ based on the idea from \cite{fu2008}.

\section{More Complex Geometry}\label{frame}

Recall that the Strominger system
\begin{eqnarray*}
F\wedge\omega^2=0,\quad F^{0,2}=F^{2,0}=0\\
\sqrt{-1}\pt\bpt\omega=\frac{\alpha'}{4}(\tr~R\wedge R-\tr~F\wedge F)\\
\ud(\|\Omega\|_\omega\cdot\omega^2)=0
\end{eqnarray*}
involves curvature forms $F$ and $R$. In this work, they will be computed with respect to the Chern connection. To do that, it is convenient to work with a holomorphic frame.

Let $e^4,e^5,e^6,e^7$ be a set of parallel orthonormal 1-forms on $T^4$ and we will use the convention \[\begin{split}\omega_I&=e^4\wedge e^5+e^6\wedge e^7,\\ \omega_J&=e^4\wedge e^6-e^5\wedge e^7,\\ \omega_K&=e^4\wedge e^7+e^5\wedge e^6.\end{split}\] In terms of this frame, it is straightforward to write down the holomorphic (3,0)-form $\Omega=\Omega_1+i\Omega_2$ where \[\begin{split}\Omega_1&=e^1\wedge\omega_I+e^2\wedge\omega_J+e^3\wedge\omega_K,\\ \Omega_2&=(-\gamma e^2+\beta e^3)\wedge\omega_I+(\gamma e^1-\alpha e^3)\wedge\omega_J+(-\beta e^1+\alpha e^2)\wedge\omega_K.\end{split}\]

Now we proceed to solve for a local holomorphic frame on $(X,J_0)$. It is easier to work with 1-forms. Consider a (1,0)-form $\theta$ of the form \[\theta=L\ud z+Ae^4+Be^5+Ce^6+De^7,\] where $z$ is a local holomorphic coordinate on $\Sigma_g$ while $L,A,B,C,D$ are smooth functions to be determined.

As $J_0\theta=i\theta$, it follows that \[\begin{split}iA&=\alpha B+\beta C+\gamma D,\\ iB&=-\alpha A+\gamma C-\beta D,\\ iC&=-\beta A-\gamma B+\alpha D,\\ iD&=-\gamma A+\beta B-\alpha C.\end{split}\] Solve $A$ and $B$ from $C$, $D$, we get \begin{equation}\label{rel2}
\begin{split}A&=-\frac{\alpha\gamma+i\beta}{\beta^2+\gamma^2}C+\frac{\alpha\beta-i\gamma}{\beta^2+\gamma^2}D=-\kappa C+\sigma D,\\ B&=\frac{\alpha\beta-i\gamma}{\beta^2+\gamma^2}C+\frac{\alpha\gamma+i\beta}{\beta^2+\gamma^2}D=\sigma C+\kappa D,\end{split}
\end{equation}
where \[\kappa=\frac{\alpha\gamma+i\beta}{\beta^2+\gamma^2}=\frac{i}{2}\left(\zeta+\frac{1}{\zeta}\right) \textrm{ and }\sigma=\frac{\alpha\beta-i\gamma}{\beta^2+\gamma^2}=-\frac{1}{2}\left(\zeta-\frac{1}{\zeta}\right)\] are holomorphic functions.

If $\theta$ is a holomorphic (1,0)-form, then \[\ud\theta=\ud L\wedge\ud z+\ud A\wedge e^4+\ud B\wedge e^5+\ud C\wedge e^6+\ud D\wedge e^7\] is of type (2,0), which is equivalent to that \[J_0(\ud\theta)=-\ud\theta.\] As a consequence, we have \[\begin{split}&(\ud A+\alpha J_0\ud B+\beta J_0\ud C+\gamma J_0\ud D)\wedge e^4+(\ud B-\alpha J_0\ud A+\gamma J_0\ud C-\beta J_0\ud D)\wedge e^5\\ +&(\ud C-\beta J_0\ud A-\gamma J_0\ud B+\alpha J_0\ud D)\wedge e^6+(\ud D-\gamma J_0\ud A+\beta J_0\ud B-\alpha J_0\ud C)\wedge e^7+2\bpt L\wedge\ud z\\ =&0.\end{split}\]

Plug in (\ref{rel2}), we can compute \[\begin{split}\ud A+\alpha J_0\ud B+\beta J_0\ud C+\gamma J_0\ud D&=-2\kappa\bpt C+2\sigma\bpt D+C(i\alpha\pt\sigma-\pt\kappa)+D(\pt\sigma+i\alpha\pt\kappa),\\ \ud B-\alpha J_0\ud A+\gamma J_0\ud C-\beta J_0\ud D&=2\sigma\bpt C+2\kappa\bpt D+C(\pt\sigma+i\alpha\pt\kappa)-D(i\alpha\pt\sigma-\pt\kappa),\\ \ud C-\beta J_0\ud A-\gamma J_0\ud B+\alpha J_0\ud D&=2\bpt C+iC(\beta\pt\kappa-\gamma\pt\sigma)-iD(\beta\pt\sigma+\gamma\pt\kappa),\\ \ud D-\gamma J_0\ud A+\beta J_0\ud B-\alpha J_0\ud C&=2\bpt D+iC(\beta\pt\sigma+\gamma\pt\kappa)-iD(\gamma\pt\sigma-\beta\pt\kappa).\end{split}\] Therefore \[\begin{split}&2\bpt L\wedge\ud z+2\bpt C\wedge(-\kappa e^4+\sigma e^5+e^6)+2\bpt D\wedge(\sigma e^4+\kappa e^5+e^7)\\ +&(C\pt\sigma+D\pt\kappa)\wedge(i\alpha e^4+e^5-i\gamma e^6+i\beta e^7)+(C\pt\kappa-D\pt\sigma)\wedge(-e^4+i\alpha e^5+i\beta e^6+i\gamma e^7)\\ =&0.\end{split}\] Each term in the above equation is a (1,1) form. Notice that \[\{\ud z, -\kappa e^4+\sigma e^5+e^6, \sigma e^4+\kappa e^5+e^7\}\] form a basis for (1,0)-forms, so we deduce that $\bpt C=\bpt D=0$ and \[2\bpt L=(C\sigma_z+D\kappa_z)(i\alpha e^4+e^5-i\gamma e^6+i\beta e^7)+(C\kappa_z-D\sigma_z)(-e^4+i\alpha e^5+i\beta e^6+i\gamma e^7),\] which is always locally solvable since the right hand side is $\bpt$-closed.

Therefore we conclude that \[\{\ud z, L_1\ud z-\kappa e^4+\sigma e^5+e^6, L_2\ud z+\sigma e^4+\kappa e^5+e^7\}\]  is a local holomorphic frame of $T^*_\ccc X$, where $L_1$ and $L_2$ are functions satisfying
\begin{equation}\label{dbar}
\begin{split}
2\bpt L_1&=\sigma_z(e^5+iJ_0e^5)-\kappa_z(e^4+iJ_0e^4)=\frac{2i\alpha_z}{\beta^2+\gamma^2}(e^7+iJ_0e^7),\\2\bpt L_2&=\kappa_z(e^5+iJ_0e^5)+\sigma_z(e^4+iJ_0e^4)=-\frac{2i\alpha_z}{\beta^2+\gamma^2}(e^6+iJ_0e^6).\end{split}
\end{equation} After taking dual basis and rescaling, we obtain a holomorphic frame of $T_\ccc X$ as follows
\begin{equation}\label{holf}
\begin{split}V_1&=i\beta e_4+i\gamma e_5+e_6-i\alpha e_7=e_6-iJ_0e_6,\\ V_2&=i\gamma e_4-i\beta e_5+i\alpha e_6+e_7=e_7-iJ_0e_7,\\ V_0&=2\frac{\pt}{\pt z}-L_1V_1-L_2V_2.\end{split}
\end{equation}
Observe that $V_1$ and $V_2$ are globally defined and nowhere vanishing. Similarly $e_4-iJ_0e_4$ and $e_5-iJ_0e_5$ are nowhere vanishing holomorphic vector fields on $X$.

At point where $(\alpha,\beta,\gamma)=(1,0,0)$, we have $V_1+iV_2=0$. Similarly at point where $(\alpha,\beta,\gamma)=(-1,0,0)$, we see $V_1-iV_2=0$. As the Gauss map is surjective, we conclude that as holomorphic vector fields, both $V_1+iV_2$ and $V_1-iV_2$ have zeroes.

In \cite[Theorem 1]{lebrun1994}, LeBrun and Simanca proved that on a compact K\"ahler manifold, the set of holomorphic vector fields with zeroes is actually a vector space. Hence we obtain a different proof that $X$ is non-K\"ahler. In fact, we can prove a little more:
\begin{prop}\label{ddbar}
$X$ does not satisfy the $\pt\bpt$-lemma. As a corollary, $X$ is not of Fujiki class $\mathcal{C}$.
\end{prop}
\begin{proof}
Let $\xi$ be a holomorphic (1,0)-form on $X$. Since its pairing with $V_1$ and $V_2$ are constants, one can easily deduce that $\xi$ must be pullback of a holomorphic (1,0)-form from $\Sigma_g$. Hence $h^0(X,\Omega^1)=h^{1,0}(X)=g$. On the other hand, $b_1(X)=2g+4>2g=2h^{1,0}(X)$, therefore the $\pt\bpt$-lemma fails.
\end{proof}

Jost and Yau \cite{jost1993} introduced the concept of astheno-K\"ahler metric, which means an Hermitian metric $\omega'$ satisfying $\pt\bpt(\omega'^{n-2})=0$. In complex dimension 3, it reads $\pt\bpt\omega'=0$, which coincides with the notion of SKT (strong K\"ahler with torsion, also known as pluriclosed) metric. An import result is the following obstruction of astheno-K\"ahler metric found by Jost-Yau:
\begin{thm}\label{ob}(Jost-Yau \cite[Lemma 6]{jost1993})\\
Let $M$ be a compact astheno-K\"ahler manifold, then every holomorphic 1-form on $M$ is closed.
\end{thm}
We shall point out that this obstruction is not enough. There exist a nilmanifold satisfying the condition of Theorem \ref{ob} which does not support any astheno-K\"ahler metric, as observed in \cite[Example 2.3]{fino2011}. This phenomenon will be manifested below as well.

A folklore conjecture says if a compact complex manifold admits both balanced and astheno-K\"ahler metrics or both balanced and SKT metrics (a priori they are different), then it must be K\"ahler. The SKT version of this conjecture has been solved in a few cases, including connected sums of $S^3\times S^3$ \cite{fu2012}, twistor spaces of anti-self-dual 4-manifolds \cite{verbitsky2014}, manifolds of Fujiki class $\mathcal{C}$ \cite{chiose2014} and nilmanifolds \cite{fino2015b, fino2015}.

To verify this conjecture for our $X$, we prove
\begin{thm}\label{folk}
$X$ is not astheno-K\"ahler/SKT.
\end{thm}
\begin{proof}
From Proposition \ref{ddbar}, we know that all the holomorphic 1-forms on $X$ are $\ud$-closed, therefore we cannot apply Theorem \ref{ob} directly. In addition, we know that $X$ is not of Fujiki class $\mathcal{C}$ from Proposition \ref{ddbar}, therefore this theorem is not covered by Chiose's result \cite{chiose2014}.

Let $\rho^j=e^j-iJ_0e^j$ for $j=4,5,6,7$. Clearly they are (1,0)-forms on $X$. Observe that \[\ud\rho^j=-i\ud(J_0e^j)\] is purely imaginary. On the other hand, $\ud\rho^j$ is of type (2,0)+(1,1), therefore we conclude that $\pt\rho^j=0$ and \[\bpt\rho^j=-i\ud(J_0e^j).\] Assume that $X$ admits an astheno-K\"ahler metric $\omega'$, then by integration by part, we have \[\int_X(\ud(J_0\rho^j))^2\wedge\omega'=\int_X\bpt\rho^j\wedge\pt\brho^j\wedge\omega'= \int_X\rho^j\wedge\brho^j\wedge\pt\bpt\omega'=0.\] On the other hand, explicit calculation shows that \[\sum_{j=4}^7(\ud(J_0\rho^j))^2=-4\ud\beta\wedge\ud\gamma\wedge\omega_I-4\ud\gamma\wedge\ud\alpha\wedge\omega_J -4\ud\alpha\wedge\ud\beta\wedge\omega_K.\] Observe that \[\frac{\ud\beta\wedge\ud\gamma}{\alpha}=\frac{\ud\gamma\wedge\ud\alpha}{\beta}=\frac{\ud\alpha\wedge\ud\beta}{\gamma} =\frac{i\ud\zeta\wedge\ud\bzeta}{(1+|\zeta|^2)^2}=G^*\omega_{\ccc\pp^1}\] is the pullback of the Fubini-Study metric by the Gauss map $G$. Therefore we have \[0=\sum_{j=4}^7\int_X(\ud(J_0\rho^j))^2\wedge\omega'= -4\int_XG^*\omega_{\ccc\pp^1}\wedge(\alpha\omega_I+\beta\omega_J+\gamma\omega_K)\wedge\omega'.\] This is in contradiction with the positivity of $\omega'$.
\end{proof}
\begin{rmk}
The result of Theorem \ref{folk} answers a question of Fu-Wang-Wu \cite[Section 1]{fu2013}. On the contrary, there exists 1-Gauduchon metrics on $X$, namely Hermitian metric $\omega'$ such that \[\pt\bpt\omega'\wedge\omega'=0.\] This is a theorem of Fu-Wang-Wu \cite[Corollary 20]{fu2013}.
\end{rmk}
\begin{rmk}
Extracting from the above argument, we can define a map \[u:H^1_{\textrm{dR}}(X;\rr)\to H^{1,1}_{\textrm{BC}}(X;\rr)\] on any compact complex manifold $(X,J)$, where $J$ is the complex structure. The map $u$ is defined to be \[u:[\rho]\mapsto[\ud(J\rho)],\] where $H^*_{\textrm{dR}}$ and $H^*_{\textrm{BC}}$ denote the de Rham cohomology and Bott-Chern cohomology respectively. If $X$ satisfies the $\pt\bpt$-lemma, then $u$ is identically zero. In general, $u$ detects the failure of $\pt\bpt$-lemma. This map $u$ is essentially related to the isomorphism \[\dfrac{H^1(X,\clo)}{H^1(X,\rr)}\cong\dfrac{\{\ud\textrm{-exact real }(1,1)\textrm{-forms}\}}{\{i\pt\bpt\psi,\psi\in C^\infty(X,\rr)\}},\] which is discussed in the proof of \cite[Proposition 1.6]{tosatti2015}.
\end{rmk}

\section{Solving Strominger System}

\subsection{Conformally Balanced Equation}

Recall from Section \ref{s2} that the naturally induced metric (\ref{metric}) \[\omega_0=\omega+\alpha\omega_I+\beta\omega_J+\gamma\omega_K\] is special balanced, therefore it solves (\ref{cb}). However, this metric is too restrictive for practical use. So we introduce a smooth function $f$ on $\Sigma_g$ and cook up a new metric \[\omega_f=e^{2f}\omega+e^f(\alpha\omega_I+\beta\omega_J+\gamma\omega_K).\] Obviously we have \[\|\Omega\|_{\omega_f}=e^{-2f}\|\Omega\|_{\omega_0}\] and \[\omega_f^2=2e^{3f}\omega\wedge(\alpha\omega_I+\beta\omega_J+\gamma\omega_K)+2e^{2f}e^4\wedge e^5\wedge e^6\wedge e^7.\] It follows that $\omega_f$ also solves the conformally balanced equation \[\ud(\|\Omega\|_{\omega_f}\omega_f^2)=0.\]

\subsection{Curvature Computation}

As we have worked out a local holomorphic frame in Section \ref{frame}, we are able to compute the term $\tr(R_f\wedge R_f)$ in (\ref{ac}) with respect to the Chern connection associated to $\omega_f$.

With respect to the local holomorphic frame $\{V_0,V_1,V_2\}$, the metric $\omega_f$ is given by the matrix \[H=2e^f\begin{pmatrix}e^f\lambda+|L_1|^2+|L_2|^2-i\alpha(L_1\bar{L}_2-L_2\bar{L}_1)&-L_1-i\alpha L_2&-L_2+i\alpha L_1\\ -\bar{L}_1+i\alpha\bar{L}_2&1&-i\alpha\\ -\bar{L}_2-i\alpha\bar{L}_1&i\alpha&1\end{pmatrix}.\] The inverse matrix can be computed accordingly \[H^{-1}=\frac{1}{2e^{2f}\lambda}\begin{pmatrix}1&L_1&L_2\\ \bar{L}_1&|L_1|^2&L_2\bar{L}_1\\ \bar{L}_2&L_1\bar{L}_2&|L_2|^2\end{pmatrix}+\frac{1}{2e^f(1-\alpha^2)}\begin{pmatrix}0&0&0\\ 0&1&i\alpha\\ 0&-i\alpha&1\end{pmatrix}.\]
Let \[\begin{split}p&=e^{2f}\lambda,\\ R&=\begin{pmatrix}1&&\\ &0&\\ &&0\end{pmatrix},\end{split}\] and set \[\begin{split}U&=\begin{pmatrix}
-L_1&-L_2\\ 1&0\\ 0&1\end{pmatrix}=\begin{pmatrix}
-L\\ \id\end{pmatrix},\\ S&=e^f\begin{pmatrix}1&-i\alpha\\ i\alpha&1\end{pmatrix},\\ V&=\begin{pmatrix}1\\ \bar{L}_1\\ \bar{L}_2\end{pmatrix}=\begin{pmatrix}1\\ \bar{L}^T\end{pmatrix}.\end{split}\]
We can express $H$ and $H^{-1}$ as \[\begin{split}H&=2pR+2US\bar{U}^T,\\ H^{-1}&=\frac{1}{2p}V\bar{V}^T+\frac{1}{2}\begin{pmatrix}0&\\&S^{-1}\end{pmatrix}.\end{split}\] Direct computation shows that \[\bar{H}^{-1}\pt\bar{H}=(p^{-1}\pt p)\begin{pmatrix}1&0\\ L^T&0\end{pmatrix}+p^{-1}\begin{pmatrix}\pt\bar{L}\\ L^T\pt\bar{L}\end{pmatrix}\begin{pmatrix}\bar{S}L^T&-\bar{S}\end{pmatrix}+\begin{pmatrix}0\\ \bar{S}^{-1}\pt \bar{S}\end{pmatrix}\begin{pmatrix}-L^T&\id\end{pmatrix}-\begin{pmatrix}0&0\\\pt L^T&0\end{pmatrix}.\] As a consequence, we have \[\begin{split}R_f=&\bpt(\bar{H}^{-1}\pt\bar{H})=(\bpt\pt\log p)\begin{pmatrix}1&0\\ L^T&0\end{pmatrix}-\pt\log p\wedge\begin{pmatrix}0&0\\ \bpt L^T&0\end{pmatrix}-\frac{\bpt p}{p^2}\begin{pmatrix}\pt\bar{L}\\ L^T\pt\bar{L}\end{pmatrix}\begin{pmatrix}\bar{S}L^T&-\bar{S}\end{pmatrix}\\ +&p^{-1}\begin{pmatrix}\bpt\pt\bar{L}\\ \bpt(L^T\pt\bar{L})\end{pmatrix}\begin{pmatrix}\bar{S}L^T&-\bar{S}\end{pmatrix}-p^{-1}\begin{pmatrix}\pt\bar{L}\\ L^T\pt\bar{L}\end{pmatrix}\begin{pmatrix}\bpt(\bar{S}L^T)&-\bpt\bar{S}\end{pmatrix}+\begin{pmatrix}0\\ \bpt(\bar{S}^{-1}\pt\bar{S})\end{pmatrix}\begin{pmatrix}-L^T&\id\end{pmatrix}\\ +&\begin{pmatrix}0\\ \bar{S}^{-1}\pt\bar{S}\end{pmatrix}\begin{pmatrix}\bpt L^T&0\end{pmatrix}-\begin{pmatrix}0&0\\ \bpt\pt L^T&0\end{pmatrix}.\end{split}\]
From this we can see immediately that \[\tr(R_f)=4\bpt\pt f.\] A even more complicated calculation reveals that \[\begin{split}\tr(R_f\wedge R_f)&=2\left[-\frac{\bpt\pt\log p}{p}\pt\bar{L}\cdot\bar{S}\cdot\bpt L^T-\frac{\pt\log p\wedge\bpt p}{p^2}\pt\bar{L}\cdot\bar{S}\cdot\bpt L^T+\frac{\pt\log p}{p}\bpt\pt\bar{L}\cdot\bar{S}\cdot\bpt L^T\right.\\ &-\frac{\pt\log p}{p}\pt\bar{L}\cdot\bpt\bar{S}\cdot\bpt L^T+\frac{\bpt p}{p^2}\pt\bar{L}\cdot\pt\bar{S}\cdot\bpt L^T-\frac{\bpt p}{p^2}\pt\bar{L}\cdot\bar{S}\cdot\bpt\pt L^T+\frac{1}{p}\pt\bar{L}\cdot\bar{S}\bpt(\bar{S}^{-1}\pt\bar{S})\bpt L^T\\ &-\frac{1}{p}\bpt\pt\bar{L}\cdot\pt\bar{S}\cdot\bpt L^T+\frac{1}{p}\bpt\pt\bar{L}\cdot\bar{S}\cdot\bpt\pt L^T+\left.\frac{1}{p}\pt\bar{L}\cdot\bpt\bar{S}\cdot\bar{S}^{-1}\pt\bar{S}\cdot\bpt L^T-\frac{1}{p}\pt\bar{L}\cdot\bpt\bar{S}\cdot\bpt\pt L^T\right].\end{split}\] Let $W=\pt\bar{L}\cdot\bar{S}\bpt L^T$. After a recombination of terms, we get a very simple expression \[\begin{split}\tr(R_f\wedge R_f)&=-\frac{2}{p}\left[(\bpt\pt\log p+\pt\log p\wedge\bpt\log p)W-\pt\log p\wedge\bpt W+\bpt\log p\wedge\pt W-\bpt\pt W\right]\\ =2\bpt\pt\left(\frac{W}{p}\right).\end{split}\] Recall that $\bpt L$ can be read off from (\ref{dbar}), hence we are able to calculate this term explicitly as \[\frac{W}{p}=-\frac{2i}{e^f\lambda}\frac{|\alpha_z|^2}{\beta^2+\gamma^2}(\alpha\omega_I+\beta\omega_J+\gamma\omega_K)= -\frac{i}{4e^f}\|\ud G\|^2(\alpha\omega_I+\beta\omega_J+\gamma\omega_K),\] where $G:\Sigma_g\to\ccc\pp^1$ is the Gauss map. Clearly it is globally defined.

\subsection{Solving the Whole System}

A crucial consequence of the lengthy calculation above is that $\tr(R_f\wedge R_f)$ is $\pt\bpt$-exact. Therefore we can simply take $F\equiv0$ to solve the Hermitian-Yang-Mills (\ref{hym}) without violating the cohomological restriction in (\ref{ac}).

We also observe that \[i\pt\bpt\omega_f=i\pt\bpt(e^f(\alpha\omega_I+\beta\omega_J+\gamma\omega_K)).\] Therefore by equating \[e^{2f}=\frac{\alpha'}{8}\|\ud G\|^2,\] we solve the whole Strominger system with $F\equiv0$.

Unfortunately $\|\ud G\|^2$ vanishes at the ramification points, at which $f$ goes to $-\infty$, thus the metric $\omega_f$ is degenerate at the fibers of $\pi$ over these ramification points. The ramification locus is unavoidable and related issues will be addressed in next section. What we get is a degenerate solution to the Strominger system.

\section{The Geometry of Degenerate Locus and Further Discussions}

A simple application of Riemann-Hurwitz formula shows that $G$ has $4(g-1)$ ramification points counted with multiplicity. These are exactly the zeroes of Gauss curvature of $\Sigma_g$. It is a very interesting question if we can find a nontrivial lower bound for number of zeroes of Gauss curvature when counting without multiplicity.

Now let us consider the case $\Sigma_g\looparrowright T^3$ is a hyperelliptic minimal surface of genus $g$. Meeks showed that $g$ must be odd. More importantly, Meeks \cite[Proposition 3.1, Theorem 3.2]{meeks1990} observed that the hyperelliptic automorphism of $\Sigma_g$ is an isometry that is induced by an inversion symmetry in $T^3$ through any hyperelliptic point. Furthermore, after a translation in $T^3$ one can manage to locate every hyperelliptic point of $\Sigma_g$ inside the set of order 2 points in $T^3$. In an ideal case with $g=3$, the 8 hyperelliptic points will be exactly the 8 order 2 points in $T^3$.

From this we see that $\|\ud G\|^2$ is invariant under the hyperelliptic automorphism $\tau$, hence our solution to the anomaly cancellation equation (\ref{ac}) descends to $X/\langle\tau\rangle$, which is a $T^4$ fibration over $\ccc\pp^1$ with orbifold singularities. However, we notice that $\tau^*\Omega=-\Omega$, therefore what we really have on $X/\langle\tau\rangle$ is a solution to the  ``twisted'' Strominger system.

In summary, on non-K\"ahler Calabi-Yau 3-folds of the type $\Sigma_g\times T^4$, we are able to write down explicit solutions to the whole Strominger system with degeneracies. A curious feature is that infinitely many topological types occur in these models. Nevertheless, there are many interesting questions left unanswered. For example, one may ask if there are any interesting physics behind the degenerate locus. If not, is it possible to apply a perturbation argument to get rid of the degeneracies? Or maybe more importantly, can one add in nontrivial $\tr(F\wedge F)$-term to fix these degeneracies? Preferably $F$ may come from anti-self-dual instantons on $T^4$. By the famous Atiyah-Ward correspondence \cite{atiyah1978}, the pullback of such instantons to $X$ have holomorphic structures satisfying the Hermitian-Yang-Mills Equation (\ref{hym}) automatically for our ansatz $\omega_f$. However, the anomaly cancellation (\ref{ac}) still awaits to be handled.

On the other hand, since the degeneracies concentrate on the fibers over ramification points, it is natural to consider the Strominger system on the twistor space instead of on the pull-back space $X$. However, a twistor space can never have trivial canonical bundle, therefore further modifications are needed. This idea will be developed in \cite{fei2015d} where local models of torsional heterotic strings are described.

\bibliographystyle{alpha}

\bibliography{C:/Users/Piojo/Dropbox/Documents/Source}

\end{document}